\documentclass[10pt,draft]{amsart}

\usepackage{amsmath,amsfonts,amssymb}
\usepackage{amsthm}

\usepackage[cp1251]{inputenc}
\usepackage[T2A]{fontenc}
\usepackage[english]{babel}

\date{}

\theoremstyle{plain}
\newtheorem{theorem}{Theorem}[section]
\newtheorem{lemma}[theorem]{Lemma}
\newtheorem{proposition}[theorem]{Proposition}

\large \numberwithin{equation}{section}

\title[Singular eigenvalue problems on the circle]
{Singular eigenvalue problems on the circle}
\address{Institute of Mathematics NAS of Ukraine \\
    Tereshchenkivska str., 3 \\
    Kyiv\\
    Ukraine\\
    01601}
\author[V. A. Mikhailets and V. M. Molyboga]
       {Volodymyr A. Mikhailets and Volodymyr M. Molyboga}
\email{mikhailets@imath.kiev.ua} \email{molyboga@imath.kiev.ua}
\begin{document}
\begin{abstract}
The eigenvalue problem on the circle for the non-self-adjoint
operators \text{ } $L_{m}(V)=(-1)^{m}\frac{d^{2m}}{dx^{2m}}+V$,
$m\in \mathbb{N}$ with singular complex-valued 2-periodic
distributions $V\in H_{per}^{-m}[-1,1]$ is studied. Asymptotic
formulae for the eigenvalues uniformly in $V$ in the space
$H_{per}^{m}[-1,1]$ and local uniformly in $V$ in the space
$H_{per}^{-m}[-1,1]$ are found.
\end{abstract}
\maketitle
\section{Introduction}\label{int}
Let introduce the complex Sobolev spaces $H_{per}^{s}[-1,1]$,
$s\in \mathbb{R}$, of 2-periodic functions or distributions. They
are defined by means their Fourier coefficients:
\begin{equation*}\label{eq_10}
  H_{per}^{s}[-1,1]:=\left\{f=\sum_{k\in\mathbb{Z}}\hat{f}(k)e^{ik\pi
  x}\left|\;\parallel
  f\parallel_{s}<\infty\right.\right\},
\end{equation*}
where
\begin{align*}\label{eq_12}
    & \parallel
  f\parallel_{s}:=\left(\sum_{k\in\mathbb{Z}}
  \langle k\rangle^{2s}\mid\hat{f}(k)\mid ^{2}\right)^{1/2},\quad \langle
  k\rangle:=1+|k|, \\
    & \hat{f}(k):=\langle f,e^{ik\pi x}\rangle, \quad k\in \mathbb{Z}.
\end{align*}
The brackets denote the sesquilinear  pairing between dual spaces
$H_{per}^{s}[-1,1]$ and $H_{per}^{-s}[-1,1]$ extending the
$L_{per}^{2}[-1,1]$-inner product
\begin{equation*}\label{eq_14}
  \langle f,g\rangle:=\frac{1}{2}\int_{-1}^{1}f(x)\overline{g(x)}\,dx, \quad
  f,g\in L_{per}^{2}[-1,1].
\end{equation*}

In the paper  we study the eigenvalue problem for the
non-self-adjoint differential operators
\begin{equation}\label{eq_16}
  L\equiv L_{m}(V):=(-1)^{m}\frac{d^{2m}}{dx^{2m}}+V,\quad m\in\mathbb{N}
\end{equation}
with the singular complex-valued distributions $V(x)\in
H_{per}^{-m}[-1,1]$. The operators $L$ are well defined as
unbounded linear operators in the Hilbert space
$H_{per}^{-m}[-1,1]$ with the inner product
\begin{equation*}\label{eq_18}
  \langle f,g\rangle_{-m}:=\sum_{k\in\mathbb{Z}}\langle
  k\rangle^{-2m}\hat{f}(k)\overline{\hat{g}(k)}
\end{equation*}
and the domain
\begin{equation*}\label{eq_20}
    D(L)=H_{per}^{m}[-1,1].
\end{equation*}
Similarly as for the functions $V(x)\in H_{per}^{0}[-1,1]$, it
turns out that the spectrum \emph{spec}(L) of the operator $L$
when considered on the interval $[-1,1]$ and with periodic
boundary conditions is discrete and consists of a sequence of the
eigenvalues
\begin{equation*}\label{eq_22}
    \lambda_{k}=\lambda_{k}(m,V), \quad k\in\mathbb{Z_{+}}
\end{equation*}
with the property that
\begin{equation}\label{eq_24}
    Re\lambda_{k}\rightarrow +\infty\quad \mbox{as} \quad k
    \rightarrow +\infty.
\end{equation}
Here the eigenvalues $\lambda_{k}$ are enumerated with their
algebraic multiplicities and ordered lexicographically, so that
\begin{equation*}\label{eq_12}
    Re(\lambda_k)<Re(\lambda_{k+1}),\quad \mbox{or} \quad
    Re(\lambda_k)=Re(\lambda_{k+1})\quad and \quad Im(\lambda_k)\leq
    Im(\lambda_{k+1}), \quad k\in\mathbb{Z_{+}}.
\end{equation*}
One can prove that operator $L$ is self-adjoint in the Hilbert
space $H_{per}^{-m}[-1,1]$ if and only if the 2-periodic
distribution $V(x)$ is real-valued, i. e.
\begin{equation*}\label{eq_8}
    \hat{V}(n)=\overline{\hat{V}(-n)}, \quad n\in\mathbb{Z}.
\end{equation*}

The aim of the paper is to find asymptotic formulae for the
eigenvalues $({\lambda_k})_{k\geq 0}$ which are uniform in $V$ on
appropriate sets of distributions. The case $m=1$ was studied in
\cite{R2, R6} using the same approach.

The following two theorems are the main results of the paper.
\begin{theorem}\label{th_10}
For any $V\in H_{per}^{-m}[-1,1]$, there exist $\varepsilon>0$,
$M=M(\varepsilon)\geq1$ and $n_{0}=n_{0}(\varepsilon)\in
\mathbb{N}$ such that for any $W\in H_{per}^{-m}[-1,1]$ with
\begin{equation*}\label{eq_25}
    \|W-V\|_{-m}\leq\varepsilon
\end{equation*}
the spectrum of the operator $L_{m}(W)$ satisfies the estimates:
\begin{itemize}
    \item [(a)] There are precisely $2n_0-1$ eigenvalues inside
    the bounded cone
    $$
    T_{M,n_0}=
    \left\{\lambda\in\mathbb{C}\,\left|\,\right.|Im\,\lambda|-M\leq Re\,\lambda \leq
    (n_{0}^{2m}-n_{0}^{m})\pi^{2m}\right\}.
    $$
    \item [(b)] For any $n\geq n_0$ the pairs of eigenvalues $\lambda_{2n-1}(W)$,
     $\lambda_{2n}(W)$ are inside a disc around $n^{2m}\pi^{2m}$ of the radius
     $n^{m}$:
     \begin{align*}\label{eq_26}
      & |\lambda_{2n-1}(W)-n^{2m}\pi^{2m}|<n^{m}, \\
      &  |\lambda_{2n}(W)-n^{2m}\pi^{2m}|<n^{m}.
     \end{align*}
\end{itemize}
\end{theorem}

In the case $m=1$ theorem \ref{th_10} has been proved in \cite{R6}
(see also \cite{R2}).
\begin{theorem}\label{th_12}
Let $R\geq 0$, $V\in H_{per}^{-m}[-1,1]$. For any $W\in
H_{per}^{-m}[-1,1]$ with
\begin{equation*}\label{eq_28}
    \|W-V\|_{m}\leq R
\end{equation*}
the eigenvalues $(\lambda_{k}(m,W))_{k\geq0}$  satisfy the
asymptotic formula
\begin{equation}\label{eq_30}
    \lambda_{2n-1}(m,W),\lambda_{2n}(m,W)=n^{2m}\pi^{2m}+o(n^{m}),\quad
    n\rightarrow\infty
\end{equation}
uniformly in $W$.
\end{theorem}

The estimates \eqref{eq_30} is novel in the case $m=1$ as well. In
the case $m=1$ and $V\in H_{per}^{k}[-1,1]$, $k\in\mathbb{Z_{+}}$
an asymptotic behavior of the eigenvalues $\lambda_{k}(1,V)$ has
been investigated in monograph \cite{R4}. The case $V\in
H_{per}^{-m\alpha}[-1,1]$, $\alpha\in [0,1)$ was earlier studied
in \cite{R5}.
\section{Preliminary Results}
A purpose of this section to prove some qualitative results
concerning the operator $L_{m}(V)$. More precisely we are going to
prove the following statement.
\begin{theorem}\label{th_14}
The operator $L_{m}(V)$, $m\in\mathbb{R}$, $V(x)\in
H_{per}^{-m}[-1,1]$ is well defined as an unbounded linear
operator in the Hilbert space $H_{per}^{-m}[-1,1]$ with the domain
\begin{equation*}\label{eq_32}
    D(L_{m})=H_{per}^{m}[-1,1].
\end{equation*}

Moreover:
\begin{enumerate}
    \item [(a)] The operator $L_{m}(V)$ is closed.
    \item [(c)] A resolvent set of the operator $L_{m}(V)$ is not
    empty and the resolvent $R(\lambda,L_{m}(V))$ is a
    compact operator.
\end{enumerate}
\end{theorem}

To prove the theorem \ref{th_14} we need two preliminary lemmas.

As well known the space $L_{per}^{2}[-1,1]$ can be isometric
identified with the sequence space $l^{2}(\mathbb{Z})$ by means of
Fourier coefficients of a function $f(x)\in H^{0}_{per}[-1,1]$.
Similarly by the Fourier transform the spaces $H^{s}_{per}[-1,1]$
identifies with sequence spaces. For any $n\in\mathbb{Z}$, and
$s\in\mathbb{R}$ we can define the weighted $l^{2}$-spaces  by
$$h^{s,n}\equiv h^{s,n}(\mathbb{Z};\mathbb{C}).$$

This space is the Hilbert space sequences $(a(k))_{k\in
\mathbb{Z}}$ in $\mathbb{C}$ with norm
\begin{equation*}\label{eq_34}
  \parallel
  a\parallel_{h^{s,n}}:=\left(\sum_{k\in\mathbb{Z}}
  \langle k+n\rangle^{2s}\mid\hat{a}(k)\mid ^{2}\right)^{1/2}.
\end{equation*}
For n=0 we will simply write $h^{s}$ instead of $h^{s,0}$.

Further, the map $$f\longmapsto(\hat{f}(k))_{k\in \mathbb{Z}}$$ is
an isometric isomorphism of the space $H^{s}_{per}[-1,1]$ onto
$h^{s}$, $s\in \mathbb{R}$. For this isomorphism, multiplication
of functions corresponds to convolution of sequences, where the
convolution product of two sequences $a=(a(k))_{k\in \mathbb{Z}}$
and $b=(b(k))_{k\in \mathbb{Z}}$ (formally) defined as the
sequence given by
\begin{equation}\label{eq_36}
    (a*b)(k):=\sum_{j\in\mathbb{Z}}a(k-j)b(j).
\end{equation}
So, given two functions $u$, $v$ formally,
\begin{equation}\label{eq_38}
    (\widehat{u\cdot
    v})(k)=\sum_{j\in\mathbb{Z}}\hat{u}(k-j)\hat{v}(j).
\end{equation}
The following Convolution Lemma is a starting point of our method.
\begin{lemma}[Convolution Lemma, \cite{R2}]\label{l_10}
Let $n\in\mathbb{Z}$, $s,r\geq0$, and $\;t\in\mathbb{R}$ with
$t\leq\min(s,r)$. If $s+r-t>1/2$, than the convolution map is
continuous (uniformly in n), when viewed as a map
    \begin{itemize}
        \item [(a)] $h^{r,n}\times h^{s,-n}\longrightarrow h^{t}$,
        \item [(b)] $h^{-t}\times h^{s,n}\longrightarrow h^{-r,n}$,
        \item [(c)] $h^{t}\times h^{-s,n}\longrightarrow h^{-r,n}$.
    \end{itemize}
\end{lemma}
So, the map
\begin{align}\label{eq_40}
    & H^{-m}_{per}[-1,1]\times H^{m}_{per}[-1,1]\longmapsto
    H^{-m}_{per}[-1,1], \\
    & (V,f)\longmapsto V\cdot f
\end{align}
is continuous, when $V\cdot f$ is given by formula \eqref{eq_34}.

Now,we can define the operator $L_{m}(V)$, which is given by the
differential expression
\begin{equation*}\label{eq_42}
  l[\cdot]:=(-1)^{m}\frac{d^{2m}}{dx^{2m}}+V,\quad m\in \mathbb{N}
\end{equation*}
with singular complex-valued potentials $V$ in
$H_{per}^{-m}[-1,1]$. The operator $L_{m}(V)$ is well defined as
an unbounded linear operator in $H_{per}^{-m}[-1,1]$ with the
dense domain $$H_{per}^{m}[-1,1].$$ Really, the derivative
operator
\begin{equation*}\label{eq_44}
(-1)^{m}\frac{d^{2m}}{dx^{2m}}: H_{per}^{m}[-1,1]\rightarrow
H_{per}^{-m}[-1,1]
\end{equation*}
and the multiplication operator $f\rightarrow V f$ maps the space
$H_{per}^{m}[-1,1]$ into $H_{per}^{-m}[-1,1]$ by \eqref{eq_36}.
\begin{lemma}\label{l_12}
The multiplication operator $V$ is $L_{m}(0)$-bounded and its
relative-bound is 0, i.e. $V\ll L_{m}(0)$.
\end{lemma}
\begin{proof}
According to the Convolution Lemma there exists the constant
$C_{m}>0$ such that
\begin{equation*}\label{eq_46}
  \parallel V u\parallel_{-m}\leq
  \begin{cases}
    C_{m}\parallel V \parallel_{m}\parallel u\parallel_{-m}, & \quad V\in H_{per}^{m}[-1,1], u\in H_{per}^{-m}[-1,1], \\
    C_{m}\parallel V \parallel_{-m}\parallel u\parallel_{m}, & \quad V\in H_{per}^{-m}[-1,1], u\in
    H_{per}^{m}[-1,1].
  \end{cases}
\end{equation*}
Further, for any fixed $\delta>0$ there exists a decomposition
$$V=V_{0}+V_{\delta}$$ with
\begin{equation*}\label{eq_48}
  V_{0}\in H_{per}^{m}[-1,1], V_{\delta}\in H_{per}^{-m}[-1,1], \parallel V_{\delta}
  \parallel_{-m}<\frac{\delta}{C_{m}}.
\end{equation*}
Taking to account that
\begin{equation*}\label{eq_50}
  \parallel u\parallel_{m}\leq \parallel u\parallel_{-m}+\parallel
  L_{m}(0)u\parallel_{-m}, \quad u\in H_{per}^{m}[-1,1]
\end{equation*}
then we have the following estimates:
\begin{align*}\label{eq_52}
 \parallel V u\parallel_{-m} & \leq\parallel V_{0}u\parallel_{-m}+\parallel
  V_{\delta}u\parallel_{-m}\leq  C_{m}\parallel V_{0} \parallel_{m}\parallel u\parallel_{-m}+
  C_{m}\parallel V_{\delta} \parallel_{-m}\parallel u\parallel_{m} \\
  & \leq\delta\parallel L_{m}(0)u\parallel_{-m}+(C_{m}\parallel V_{0} \parallel_{m}+\delta)\parallel
  u\parallel_{-m}.
\end{align*}
Hence $V\ll L_{m}(0)$.
\end{proof}

Now we can prove Theorem \ref{th_14}. Statement (a) is a
consequence of Lemma \ref{l_12} and Theorem 1.11 (\cite{R3},Ch.
IV) since the operator $L_{m}(0)$ is self-adjoint. Similarly
statement (c) is a consequence of Lemma \ref{l_12} and Theorem
3.17 (\cite{R3}, Ch. IV) since a resolvent $R(\lambda, L_{m}(0))$
is a compact operator.

Remark that using the perturbation results (\cite{R1}, Ch. V, \S
11) and \cite{R3,R7} one can prove the following statement
\begin{theorem}\label{th_16}
For any $\varepsilon>0$ the spectrum of the operator $L_{m}(V)$
belongs to the cone
\begin{equation*}\label{eq_54}
  S_{\varepsilon}:=\left\{\lambda\in\mathbb{C}\left||\arg\lambda|<\varepsilon\right.\right\}
\end{equation*}
except a finite number of the eigenvalues. The asymptotic formula
\begin{equation*}\label{eq_56}
  \lambda_{n}(m,V)\sim\lambda_{n}(m,0), \quad n\rightarrow\infty
\end{equation*}
is valid.
\end{theorem}

Obviously that assertions of Theorems \ref{th_10} and \ref{th_12}
are much stronger. Therefore the proof of Theorem \ref{th_16} is
omitted.
\section{Proofs of the Main Theorems}
To prove the theorems \ref{th_10} and \ref{th_12} it is useful to
deal with the eigenvalue problem for the operator $\hat{L}_{m}(v)$
in the sequence Hilbert space $h^{-m}(\mathbb{Z})$. This operator
has the same spectrum and is of the form
\begin{equation*}\label{eq_58}
    \hat{L}_{m}\equiv\hat{L}_{m}(v)=D_{m}+B(v),
\end{equation*}
where $D_{m}$ and $B(v)$ are infinite matrices,
\begin{align*}\label{eq_60}
 & D_{m}(k,j):=k^{2m}\pi^{2m}\delta_{kj}, \\
 & B(k,j):=v(k-j), \quad k,j\in \mathbb{Z}
\end{align*}
with $v(k):=\hat{V}(k)$ in the  sequence space $h^{-m}$. By the
Convolution Lemma $$B(v)a=v*a$$ is well defined for $a\in h^{m}$
and hence the operator $\hat{L}_{m}(v)=D_{m}+B(v)$ is well defined
as an unbounded linear operator in the Hilbert space $h^{-m}$ with
the domain $$D(\hat{L}_{m})=h^{m}.$$

The eigenvalue problem
\begin{equation*}\label{eq_62}
  \hat{L}_{m}(v)f=\lambda f, \quad f\in h^{m}
\end{equation*}
is studied in this section. For this purpose we will compare the
spectrum $spec(D_{m}+B(v))$ of the operator
$\hat{L}_{m}(v)=D_{m}+B(v)$ to the spectrum of the operator
$D_{m}$. It is clearly that
\begin{equation*}
  spec(D_{m})=\{k^{2m}\pi^{2m}|k\in\mathbb{Z}_{+}\},
\end{equation*}
where the eigenvalue $0$ is simple and other eigenvalues are
double.

For given $M\geq 1$, $n\geq 1$, and $0<r_{n}<n^{m}\pi^{2m}$ the
following regions $Ext_{M}$ and $Vert^{m}_{n}(r)$,
$m\in\mathbb{N}$ of complex plane will be used:
\begin{align*}\label{eq_64}
    & Ext_{M}:=\{\lambda\in\mathbb{C}\;|\;Re\,\lambda\leq|Im\,\lambda|-M\},\\
    & Vert^{m}_{n}(r_{n}):=\{\lambda=n^{2m}\pi^{2m}+z\in\mathbb{C}\;|\;|Re\,z| \leq
  n^{m}\pi^{2m},\;|z|\geq r_{n}\}.
\end{align*}
Let formulate the result which we will use bellow.
\begin{proposition}[\cite{R5}]\label{pr_10}
Let $R>0$. There exist $\mathrm{M}\geq1$ and $n_{0}\in\mathbb{N}$
so that for any $v\in h^{0}$ with $$\|v\|_{h^{0}}\leq\mathrm{R}$$
the spectrum \emph{spec}$(D_{m}+B(v))$ of the operator
$\hat{L}_{m}(v)=D_{m}+B(v)$ with $B(v)=v\ast\cdot$ consists of a
sequence $(\lambda_{k}(v))_{k\geq 0}$ such that:
\begin{itemize}
  \item [(a)] There are precisely $2n_0-1$ eigenvalues
  inside the bounded cone
\begin{equation*}\label{eq_68}
  T_{M,n_0}=\left\{\lambda\in\mathbb{C}\,\left|\,\right.|Im\,\lambda|-M\leq Re\,\lambda \leq
    (n_{0}^{2m}-n_{0}^{m})\pi^{2m}\right\}.
\end{equation*}
  \item [(b)] For $n\geq n_0$ the pairs of eigenvalues
  $\lambda_{2n-1}(m,v)$, $\lambda_{2n}(m,v)$ are inside a
  disc around $n^{2m}\pi^{2m}$,
\begin{align*}\label{eq_70}
    & |\lambda_{2n-1}(v)-n^{2m}\pi^{2m}|<(3^{m}\sqrt{2}+1)\mathrm{R},\\
    & |\lambda_{2n}(v)-n^{2m}\pi^{2m}|<(3^{m}\sqrt{2}+1)\mathrm{R}.
\end{align*}
\end{itemize}
\end{proposition}
In a straightforward way, one can prove the following two
auxiliary lemmas.
\begin{lemma}\label{l_16}
For any $s,t\in\mathbb{R}$ with $s-t\leq2$ and any
$\lambda\in\mathbb{C}\setminus spec(D_{m})$, $m\in\mathbb{N}$ we
have $(\lambda-D_{m})^{-1}\in\mathcal{L}(h^{mt},h^{ms})$ with norm
\begin{equation*}\label{eq_72}
  \parallel(\lambda-D_{m})^{-1}\parallel_{\mathcal{L}(h^{mt},h^{ms})}=
  \sup_{k\in\mathbb{Z}}\frac{<k>^{m(s-t)}}{\mid\lambda-k^{2m}\pi^{2m}\mid}<\infty.
\end{equation*}
\end{lemma}
\begin{lemma}\label{l_18}
Uniformly for $n\in\mathbb{Z}\setminus\{0\}$ and $\lambda\in
Vert^{m}_{n}(r_{n})$
\begin{align*}
 (a)& \quad\parallel(\lambda-D_{m})^{-1}\parallel_{\mathcal{L}(h^{-m})}=\frac{1}{r_{n}}O(1),
  & \quad (a') & \quad\parallel(\lambda-D_{m})^{-1}\parallel_{\mathcal{L}(h^{-m})}=O(n^{-m}), \\
  (b) & \quad\parallel(\lambda-D_{m})^{-1}\parallel_{\mathcal{L}(h^{-m,n})}=\frac{1}{r_{n}}O(1),
  & \quad
  (b')&\quad\parallel(\lambda-D_{m})^{-1}\parallel_{\mathcal{L}(h^{-m,n})}=O(n^{-m}),\\
  (c)& \quad\parallel(\lambda-D_{m})^{-1}\parallel_{\mathcal{L}(h^{-m,n},h^{-m})}=\frac{1}{r_{n}}O(1),
  & \quad (c') &
  \quad\parallel(\lambda-D_{m})^{-1}\parallel_{\mathcal{L}(h^{-m,n},h^{-m})}=O(n^{-m}),
  \\
  (d)& \quad\parallel(\lambda-D_{m})^{-1}\parallel_{\mathcal{L}(h^{-m},h^{m,n})}=\frac{n^{2m}}{r_{n}}O(1),
  & \quad (d') &
  \quad\parallel(\lambda-D_{m})^{-1}\parallel_{\mathcal{L}(h^{-m},h^{m,n})}=O(n^{m}),
  \\
  (e)&
  \quad\parallel(\lambda-D_{m})^{-1}\parallel_{\mathcal{L}(h^{-m,n},h^{m,-n})}=\frac{n^{2m}}{r_{n}}O(1);
  & \quad (e') &
  \quad\parallel(\lambda-D_{m})^{-1}\parallel_{\mathcal{L}(h^{-m,n},h^{m,-n})}=O(1).
\end{align*}
\end{lemma}
\begin{proposition}\label{pr_12}
Let $v\in h^{-m}$. There exist $\varepsilon>0$, $M\geq1$ and
$n_{0}\in \mathbb{N}$ (both depending on $\varepsilon>0$) so that
for any $w\in h^{-m}$ with $$\parallel
w-v\parallel_{h^{-m}}\leq\varepsilon$$ the spectrum
\emph{spec}$(D_{m}+B(w))$ of the operator
$$\hat{L}_{m}(w)=D_{m}+B(w)$$ with $B(w)=w\ast\cdot$ consists of a
sequence $(\lambda_{k}(m,w))_{k\geq 0}$ such that:
\begin{itemize}
  \item [(a)] There are precisely $2n_0-1$ eigenvalues
  inside the bounded cone
\begin{equation*}\label{eq_74}
  T_{M,n_0}=\left\{\lambda\in\mathbb{C}\,\left|\,\right.|Im\,\lambda|-M\leq Re\,\lambda \leq
    (n_{0}^{2m}-n_{0}^{m})\pi^{2m}\right\}.
\end{equation*}
  \item [(b)] For $n\geq n_0$ the pairs of eigenvalues
  $\lambda_{2n-1}(m,w)$, $\lambda_{2n}(m,v)$ are inside a
  disc around $n^{2m}\pi^{2m}$,
\begin{equation*}\label{eq_76}
  |\lambda_{2n-i}(m,w)-n^{2m}\pi^{2m}|<n^{m},
  \quad i=0,1.
\end{equation*}
\end{itemize}
\end{proposition}
\begin{proof}
Let $v\in h^{-m}$. Since the set $h^{m}$ is dense in the space
$h^{-m}$, we can represent $v$ in the form $$v=v_{0}+v_{1},$$ with
$v_{0}\in h^{m}$ and $\parallel
v_{1}\parallel_{h^{-m}}\leq\varepsilon$, where $\varepsilon>0$
will be find bellow. We will show that for some $M\geq 1$ and
$n_{0}\in\mathbb{N}$, which both depending on $\parallel
v_{0}\parallel_{h^{m}}$, so that for any $w=v+\tilde{w}\in h^{-m}$
with $\parallel \tilde{w}\parallel_{h^{-m}}\leq\varepsilon$,
\begin{equation}\label{eq_80}
  Ext_{M}\cup\bigcup_{n\geq
  n_{0}}Vert^{m}_{n}(n^{m})\subseteq\emph{Resol}
  (\hat{L}_{m}(w)),
\end{equation}
where $\emph{Resol}(\hat{L}_{m}(w))$ denotes the resolvent set of
the operator $\hat{L}_{m}(w)=D_{m}+B_{0}+B_{1}$, and
$B_{0}=v_{0}\ast\cdot$, and $B_{1}=(v_{1}+\tilde{w})\ast\cdot$.

At first let consider $\lambda\in Ext_{M}$ for $M\geq 1$. Using
the Convolution Lemma and the Lemma \ref{l_16} one gets
\begin{equation*}\label{eq_82}
  \parallel
  B_{0}(\lambda-D_{m})^{-1}\parallel_{\mathcal{L}(h^{-m})}
  \leq C_{m}\parallel v_{0}\parallel_{h^{m}}
  \parallel(\lambda-D_{m})^{-1}\parallel_{\mathcal{L}(h^{-m})} =\parallel
  v_{0}\parallel_{h^{m}}\cdot O(M^{-1}).
\end{equation*}
Hence, for $M\geq 1$ large enough and $\lambda\in Ext_{M}$,
\begin{equation*}\label{eq_84}
  L_{\lambda}:=\lambda-D_{m}-B_{0}
  =(Id-B_{0}(\lambda-D_{m})^{-1}) (\lambda-D_{m})
\end{equation*}
is invertible in $\mathcal{L}(h^{-m})$ with inverse
\begin{equation}\label{eq_86}
  L_{\lambda}^{-1}=(\lambda-D_{m})^{-1}
  (Id-B_{0}(\lambda-D_{m})^{-1})^{-1}.
\end{equation}
So, using the Convolution Lemma and the estimate
$$\parallel(\lambda-D_{m})^{-1}\parallel_{\mathcal{L}
(h^{-m},h^{m})}=O(1),$$ we obtain
\begin{equation*}\label{eq_88}
  \parallel B_{0}L_{\lambda}^{-1} \parallel_{\mathcal{L}(h^{-m})}
  =O(\parallel(v_{1}+\tilde{w})\parallel_{h^{-m}}=O(\varepsilon).
\end{equation*}
Therefore, if $\varepsilon> 0$ is small enough, the resolvent of
the operator $$\hat{L}_{m}(w)=D_{m}+B_{0}+B_{1}$$ exists in the
space $\mathcal{L}(h^{-m})$ for $\lambda\in Ext_{M}$ and is given
by the formula
\begin{equation}\label{eq_90}
  (\lambda-D_{m}-B_{0}-B_{1})^{-1}=(L_{\lambda}^{-1}-B_{1})^{-1}
  =L_{\lambda}^{-1}\Sigma_{k\geq 0}(B_{1}L_{\lambda}^{-1})^{k}.
\end{equation}
Consequently, for $M$ large enough, $Ext_{M}\subseteq
\emph{Resol}(\hat{L}_{m}(w))$.

To treat $\lambda\in Vert^{m}_{n}(n^{m})$, first note that,
unfortunately, $$\parallel(\lambda-D_{m})^{-1}
\parallel_{\mathcal{L}(h^{-m},h^{m})}=O(n^{m}),$$ and so
we can not argue as above. However, we have (see the Lemma
\ref{l_18} $(e')$)
\begin{equation}\label{eq_92}
\parallel(\lambda-D_{m})^{-1}
\parallel_{\mathcal{L}(h^{-m,n},h^{m,-n})}=O(1).
\end{equation}
Now, for $\lambda\in Vert^{m}_{n}(n^{m})$ with $n$ large enough,
we find that the following decomposition of the resolvent of
$\hat{L}_{m}(w)=D_{m}+B_{0}+B_{1}$ converges in the space
$\mathcal{L}(h^{-m})$,
\begin{equation}\label{eq_94}
  (\lambda-D_{m}-B_{0}-B_{1})^{-1}=L_{\lambda}^{-1}
  +L_{\lambda}^{-1}T_{\lambda} (B_{1}L_{\lambda}^{-1}) +L_{\lambda}^{-1}T_{\lambda}
  (B_{1}L_{\lambda}^{-1})^{2},
\end{equation}
where $L_{\lambda}=\lambda-D_{m}-B_{0}$, and
$T_{\lambda}:=\Sigma_{l\geq 0} (B_{1}L_{\lambda}^{-1})^{2l}$ is
considered as an element in $\mathcal{L}(h^{-m,n})$. Using the
Convolution Lemma (c) and the Lemma \ref{l_18} ($a'$), ($b'$) we
can find $n_{0}\in\mathbb{N}$ such that, for any $n\geq n_{0}$ and
$\lambda\in Vert^{m}_{n}(n^{m})$, the operator $L_{\lambda}$ is
invertible in the spaces $\mathcal{L}(h^{-m})$ and
$\mathcal{L}(h^{-m,n})$ in the form \eqref{eq_86}. Using the
Convolution Lemma (a) and the Lemma \ref{l_18} ($e'$), one can
obtain
\begin{equation*}\label{eq_96}
  \parallel B_{1}L_{\lambda}^{-1}
  \parallel_{\mathcal{L}(h^{-m,n},h^{-m,n})}\leq C_{m}\parallel
  (v_{1}+\tilde{w})\parallel_{h^{-m}}
  \parallel L_{\lambda}^{-1}\parallel_{\mathcal{L}
  (h^{-m,n},h^{-m,n})}=O(\varepsilon).
\end{equation*}
Therefore, if $\varepsilon> 0$ is small enough, the sum
$$T_{\lambda}=\Sigma_{l\geq 0} (B_{1}L_{\lambda}^{-1})^{2l}$$
converges in $\mathcal{L}(h^{-m,n})$. Then the representation
\eqref{eq_94} follows because $B_{1}L_{\lambda}^{-1}\in
\mathcal{L}(h^{-m},h^{-m,n})$ by the Convolution Lemma (a)  and
the Lemma \ref{l_18} ($d'$), and $L_{\lambda}^{-1}\in
\mathcal{L}(h^{-m,n},h^{-m})$ by the Lemma \ref{l_18} ($c'$).

Hence, for $\varepsilon\geq 0$, $M\geq 1$, and $n_{0}\in
\mathbb{N}$ as above, the inclusion \eqref{eq_80} holds. Let
remark, that in fact, we have proved the inclusion
\begin{equation}\label{eq_98}
  Ext_{M}\cup\bigcup_{n\geq
  n_{0}}Vert^{m}_{n}(n^{m})\subseteq\emph{Resol}
  (\hat{L}_{m}(w(s))),
\end{equation}
where $\emph{Resol}(\hat{L}_{m}(w))$ denotes the resolvent set of
the operator $\hat{L}_{m}(w)=D_{m}+B_{0}+sB_{1}$ for $0\leq s\leq
1$. Further, denoting the Riesz projector for $0\leq s\leq 1$ and
any contour $\Gamma$ in $Ext_{M}\cup\bigcup_{n\geq
n_{0}}Vert^{m}_{n}(n^{m})$,
\begin{equation*}\label{eq 80}
  P(s):=\frac{1}{2\pi i}\int_{\Gamma}
  (\lambda-D_{m}-B_{0}-sB_{1})^{-1}\,d\lambda\in
  \mathcal{L}(h^{-m}),
\end{equation*}
one concludes that the operators $D_{m}+B_{0}$ and
$D_{m}+B_{0}+B_{1}$ have the same number of eigenvalues (counted
with their algebraic multiplicity) inside $\Gamma$. To complete
the proof of Proposition \ref{pr_12} it is sufficient to apply
Proposition \ref{pr_10} to the operator $D_{m}+B_{0}$ with
$B_{0}=v_{0}\ast\cdot$ and $v_{0}\in h^{m}\subseteq h^{0}$.
\end{proof}
\begin{proposition}\label{pr_14} Consider the eigenvalue
problem $\hat{L}_{m}(v)f=\lambda f$ with $v$ in $h^{-m}$, and let
$R\geq 0$. For any $w\in h^{-m}$ with $$\parallel
w-v\parallel_{h^{m}}\leq R$$ the spectrum
\emph{spec}$(D_{m}+B(w))$ of the operator
$\hat{L}_{m}(w)=D_{m}+B(m)$ with $B(w):=w\ast\cdot$ consists of a
sequence $(\lambda_{k}(w))_{k\geq 0}$ ordered lexicographically of
complex-valued eigenvalues counted with their algebraic
multiplicity satisfies the asymptotic formula
\begin{equation*}\label{eq 102}
    \lambda_{2n-i}(m,w)=n^{2m}\pi^{2m}+o(n^{m}),\quad i=0,1,\quad
    n\rightarrow\infty
\end{equation*}
holds.
\end{proposition}
\begin{proof}
Let $v\in h^{-m}$. Since the set $h^{m}$ is dense in the space
$h^{-m}$, one decomposes $$v=v_{0}+v_{1},$$ with $v_{0}\in h^{m}$
and $\parallel v_{1}\parallel_{h^{-m}}\leq\varepsilon$, where
$\varepsilon>0$ will be chosen bellow. We are going to show as
above that there exists $n_{0}\in\mathbb{N}$ depending on
$\parallel v_{0}\parallel _{h^{m}}$ and $R\geq 0$ such that, for
any $w=v+w_{0}\in h^{-m}$ with $\parallel w_{0}\parallel_{h^{-m}}
\leq R$,
\begin{equation}\label{eq_104}
  \bigcup_{n\geq n_{0}}Vert^{m}_{n}(n^{m}) \subseteq\emph{Resol}
  (\hat{L}_{m}(w)),
\end{equation}
where $\emph{Resol}(\hat{L}_{m}(w))$ denotes the resolvent set of
the operator $\hat{L}(w)=D_{m}+B_{0}+B_{1}$, and $B_{0}=
(v_{0}+w_{0}) \ast \cdot$, and $B_{1}=v_{1}\ast\cdot$. Notice,
that now we consider the strips $Vert^{m}_{n}(r_{n})$ with
$r_{n}=\delta n^{m}$ for some $\delta\in (0,1]$.

So, let $\lambda\in Vert^{m}_{n}(r_{n})$. Using the Convolution
Lemma and the Lemma \ref{l_18} (b) one gets
\begin{equation*}\label{eq_106}
  \parallel
  B_{0}(\lambda-D_{m})^{-1}\parallel_{\mathcal{L}(h^{-m,n})}
  \leq C_{m}\parallel (v_{0}+w_{0})\parallel_{h^{m}}
  \parallel(\lambda-D_{m})^{-1}\parallel_{\mathcal{L}(h^{-m,n})} =\frac{\parallel
  (v_{0}+w_{0})\parallel_{h^{m}}}{r_{n}}O(1).
\end{equation*}
Hence, for $n$ large enough and $\lambda\in Vert^{m}_{n}(r_{n})$,
\begin{equation*}\label{eq_108}
  L_{\lambda}:=\lambda-D_{m}-B_{0}
  =(Id-B_{0}(\lambda-D_{m})^{-1}) (\lambda-D_{m})
\end{equation*}
is invertible in $\mathcal{L}(h^{-m})$ with inverse
\begin{equation}\label{eq_110}
  L_{\lambda}^{-1}=(\lambda-D_{m})^{-1}
  (Id-B_{0}(\lambda-D_{m})^{-1})^{-1}.
\end{equation}
Further, for $n$ large enough, we can show that the following
representation  of  resolvent of the operator $\hat{L}_{m}(w)=
D_{m}+B_{0}+B_{1}$ converges in $\mathcal{L}(h^{-m})$,
\begin{equation}\label{eq_112}
  (\lambda-D_{m}-B_{0}-B_{1})^{-1}=L_{\lambda}^{-1}
  +L_{\lambda}^{-1}T_{\lambda} (B_{1}L_{\lambda}^{-1}) +L_{\lambda}^{-1}T_{\lambda}
  (B_{1}L_{\lambda}^{-1})^{2},
\end{equation}
where $L_{\lambda}=\lambda-D_{m}-B_{0}$, and
$T_{\lambda}:=\Sigma_{l\geq 0} (B_{1}L_{\lambda}^{-1})^{2l}$ is
considered as an element in $\mathcal{L}(h^{-m,n})$. Using the
Convolution Lemma  and the Lemma \ref{l_18} (e), we get
\begin{equation*}\label{eq_114}
  \parallel B_{1}L_{\lambda}^{-1}
  \parallel_{\mathcal{L}(h^{-m,n},h^{-m,n})}\leq C_{m}\parallel
  v_{1}\parallel_{h^{-m}} \parallel L_{\lambda}^{-1}\parallel _{\mathcal{L}
  (h^{-m,n},h^{-m,n})}=O(\varepsilon).
\end{equation*}
Hence, if $\varepsilon> 0$ is small enough, the sum
$T_{\lambda}=\Sigma_{l\geq 0} (B_{1}L_{\lambda}^{-1})^{2l}$
converges in the space $\mathcal{L}(h^{-m,n})$ and the
representation \eqref{eq 102} then follows because
$B_{1}L_{\lambda}^{-1}\in \mathcal{L}(h^{-m},h^{-m,n})$ by the
Convolution Lemma  and the Lemma \ref{l_18} (d), and
$L_{\lambda}^{-1}\in \mathcal{L} (h^{-m,n}, h^{-m})$ by the Lemma
\ref{l_18} (c).

Consequently, for some $\varepsilon\geq 0$ and $n_{0}\in
\mathbb{N}$ he inclusion \eqref{eq_80} holds for $r_{n}=\delta
n^{m}$, $\delta\in (0,1]$. Defining the Riesz projector for $0\leq
s\leq 1$ and any contour $\Gamma$ in $\bigcup_{n\geq n_{0}}
Vert^{m}_{n} (n^{m})$,
\begin{equation*}\label{eq 116}
  P(s):=\frac{1}{2\pi i}\int_{\Gamma}
  (\lambda-D_{m}-B_{0}-sB_{1})^{-1}\,d\lambda\in
  \mathcal{L}(h^{-m}),
\end{equation*}
one concludes that the operators $D_{m}+B_{0}$ and
$D_{m}+B_{0}+B_{1}$ have the same number of eigenvalues (counted
with their algebraic multiplicity) inside $\Gamma$. Applying
Proposition \ref{pr_12} to the operator $D_{m}+B_{0}$ with
$B_{0}=(v_{0}+w_{0})\ast\cdot$ and $v_{0}+w_{0}\in h^{m}\subseteq
h^{0}$ one gets:

the spectrum \emph{spec}$(D_{m}+B(w))$ of the operator
$\hat{L}_{m}(w)=D_{m}+B(w)$ with $B(w)=w\ast\cdot$ consists of a
sequence $(\lambda_{k}(m,w))_{k\geq 0}$ of complex-valued
eigenvalues, and for any $\delta\in (0,1]$ there exists
$n_{0}\in\mathbb{N}$ such that the pairs of eigenvalues
$\lambda_{2n-1}(m,w)$, $\lambda_{2n}(m,w)$ there are inside a disc
around $n^{2m}\pi^{2m}$,
\begin{equation*}\label{eq_118}
    |\lambda_{2n-i}(w)-n^{2m}\pi^{2m}|<\delta n^{m},\quad i=0,1.
\end{equation*}
So, we conclude that the sequence $(\lambda_{k}(m,w))_{k\geq 0}$
of eigenvalues satisfies the asymptotic formula \eqref{eq_30}.
\end{proof}
\section{Acnowledgement}
The first author (V.A.M.) was partially supported by NFBR of
Ukraine under Grants 01.07/027 and 01.07/00252.

\end{document}